\documentclass[article,11pt]{amsart}
\input{Macros}
\usepackage[english]{babel}
\usepackage{tabu}
\usepackage{amssymb}
\usepackage{amscd}
\usepackage[abbrev,alphabetic]{amsrefs}
\usepackage{xypic}

\DeclareMathOperator{\aC}{C_a}
\DeclareMathOperator{\aBC}{BC_a}

\title[Normality of minimal lc centers on threefolds]{Normality of minimal log canonical centers of threefolds in mixed and positive characteristic}
\author{Emelie Arvidsson and Quentin Posva}
\date{}

\address{University of Utah, Department of Mathematics, JWB 311, 84112 Salt Lake City, Utah, USA}
\email{arvidsson@math.utah.edu}

\address{University of Utah, Department of Mathematics, JWB 311, 84112 Salt Lake City, Utah, USA}
\email{posva@math.utah.edu}

\begin{document}

\maketitle

\begin{quote}
\textsc{Abstract.} We prove the normality of minimal log canonical centers on threefold pairs which residue fields are perfect of residue characteristics $p\neq 2,3 $ and $5$. We also show that the union of all log canonical centers on threefold pairs with standard coefficients are seminormal provided that the residue characteristic is large enough.  We provide an example of a non-seminormal log canonical center on a threefold in characteristic $3$, and give sufficient conditions to construct similar examples.
\end{quote}

\tableofcontents

\section{Introduction}
Log canonical centers play a central role in the birational geometry of varieties. The worst singularities of a log canonical variety are concentrated at its log canonical centers, and the understanding of those is oftentimes a key step to the understanding of the whole variety. In particular, the cohomological and topological properties of log canonical centers are worth investigating.

General properties of log canonical centers are well-understood in characteristic $0$: see for example \cite[\S 4.3]{Kollar_Singularities_of_the_minimal_model_program}. It is known that in characteristic $0$ log canonical centers are always seminormal and that the minimal ones are actually normal. These properties can be established in two steps. First, one studies the log canonical centers of dlt pairs, which are particularly simple \cite[\S 4.2]{Kollar_Singularities_of_the_minimal_model_program}. Then one relates the log canonical centers on lc pairs to the ones on dlt pairs, using dlt resolutions and some vanishing theorems.

The situation in positive characteristic and mixed characteristic is more complicated and not as well-understood. In particular, some crucial vanishing theorems are simply false. Without restriction on the characteristics, even log canonical centers on dlt pairs need not be well-behaved \cite{Cascini_Tanaka_Plt_threefolds_with_non_normal_centres_in_char_2, Bernasconi_Non_normal_plt_centers_in_pos_char}. While the surface theory is quite similar in equicharacteristics $0$ and $p>0$ and in mixed characteristic \cite[\S 2.2, \S 3.3]{Kollar_Singularities_of_the_minimal_model_program}, already for threefolds significant differences appear.

Let us give a summary of the current understanding of log canonical centers in positive and mixed characteristics.
    \begin{itemize}
        \item \textbf{Plt centers on threefolds.} Normality in characteristic $>5$ is shown in \cite[Theorem 3.11, Proposition 4.1]{Hacon_Xu_On_the_3dim_MMP_in_pos_char}. Without restriction on the characteristic, plt centers on threefolds are known to be normal up to universal homeomorphism \cite[Theorem 1.2]{Hacon_Witaszek_On_the_relative_MMP_for_threefolds_in_low_char} (see also \cite[\S 3.2]{Gongyo_Nakamura_Tanaka_Rational_pts_on_Fano_threefolds}). This is the best we can hope for in general, since in \cite{Cascini_Tanaka_Plt_threefolds_with_non_normal_centres_in_char_2} an example of non-normal plt center on a threefold in characteristic $2$ is given.
        \item \textbf{Plt centers in arbitrary dimensions.} For every prime $p>0$ we have examples of non-normal plt centers in characteristic $p$ and dimension $2p+2$ \cite{Bernasconi_Non_normal_plt_centers_in_pos_char}.
        \item \textbf{Minimal log canonical centers on threefolds.} 
        Normality up to universal homeomorphism in mixed characteristic different from $2,3$ and $5$ follows from \cite{Filipazzi_Waldron_Connectedness_principle_char>5}. For minimal log canonical centers on semi-log canonical threefolds in characteristic $>5$, this is also proved in \cite[Theorem 7]{Posva_Gluing_for_surfaces_and_threefolds}. Minimal log--canonical centers of dlt--pairs are normal in equicharacteristic $>5$, as well as in mixed characteristic of residue characteristic $>5$ \cite[Theorem 3.1]{Bernasconi_Kollar_vanishing_theorems_for_threefolds_in_char_>5}
    \end{itemize}

 One could hope that in a fixed dimension, log canonical centers satisfy (semi-)normality properties after bounding from below the residue characteristics. This hope falls short: recently Koll\'ar \cite{Kollar22} constructed in every positive characteristic examples of non-weakly normal log canonical centers on fourfold pairs. Still, the question remains open for unions of log canonical centers on threefolds. The main result of this note is the normality of minimal ones in mixed and positive characteristics different from $2,3$ and $5$:

\begin{theorem_intro}[\autoref{maintheorem}]\label{theorem_intro:main}
Let $(X,\Delta)$ be a log canonical threefold whose closed points have perfect residue fields of characteristic different from $2,3 $ and $5$. Then the minimal log canonical centers of $(X,\Delta)$ are normal.
\end{theorem_intro}

If a minimal lc center is a surface, then the pair is plt around that surface and normality follows from the aforementioned works. So the only case left to study is when the minimal lc center is a curve, say $C\subset X$. We follow the strategy used in characteristic $0$. Indeed, when the residue characteristics are greater than five, threefold dlt singularities have all the desired properties. We would like to descend these properties from dlt modifications to log canonical pairs. In order to do so we will need some vanishing theorem. The sufficient vanishing theorem is established in \cite{Arv22}. More precisely, given a threefold log canonical singularity $(X,\Delta)$ with minimal log canonical center a curve $C$, by \cite[Prop. 14]{Arv22} there exists a dlt model $g\colon Y\to X$ with exceptional divisor $E$ contracting to $C$, such that $E$ is weakly normal and $R^1g_*\sO(-E)=0$. Combining this vanishing with a factorization of $E\to C$ through the normalization of $C$, whose existence is guaranteed by \autoref{prop:factorization_in_weakly_normal_case}, we obtain that $C$ is normal.

Building on \cite[Prop. 15]{Arv22} and on the recent work of Kawakami on Kawamata--Viehweg vanishing on Calabi--Yau surfaces \cite{Kawakami_KVV_for_log_CY_surfaces_in_large_char}, we are also able to prove some semi-normality property, provided that the characteristic is large enough:

\begin{theorem_intro}[\autoref{Mainsemi}]\label{theorem_intro:semi_normality}There exists an integer $p_0$ such that if $\Delta$ has standard coefficients and $(X,\Delta)$ is a log canonical threefold whose closed points have perfect residue fields of characteristic $p>p_0$,
then the union of all the log canonical centers of $(X,\Delta)$ is semi-normal.
\end{theorem_intro}

Whether an arbitrary union of log canonical centers is semi-normal is not a formal consequence of this theorem, even in characteristic $0$ (see for example the proof given in \cite[Sections 4.3-4.4]{Kollar_Singularities_of_pairs}).

Also, for threefold pairs whose reduced boundary is $\mathbb{Q}$-Cartier with index not divisible by any of the residue characteristics, it is possible to show that the boundary is demi-normal as long as the residue characteristics at closed points are different from $2,3$ and $5$. This can be obtained by combining the main theorem of \cite{ABP} and cyclic covers (see also \cite[Cor. 23]{Arv22}). We include it here in \autoref{prop:demi_normality} for completeness.



Let us now discuss counter-examples to (semi)-normality properties.
The only counter-example currently known to normality of plt centers for threefolds, obtained in \cite{Cascini_Tanaka_Plt_threefolds_with_non_normal_centres_in_char_2}, stems from a failure of Kawamata--Viehweg vanishing on some del Pezzo surface. After taking a cone over that surface, the authors obtain a plt pair whose boundary is not $S_2$ at the vertex. Kawamata--Viehweg vanishing holds on del Pezzo surfaces in characteristic $>5$ by \cite{ABL}, but we have counter-examples in characteristic $3$ \cite{Bernasconi_Kawamata-Viehweg_vanishing_fails} and in characteristic $5$ by \cite{ABL}. (See also \cite{Nagaoka_Non_liftable_log_dP_in_char_5} for the relation between Kawamata--Viewheg vanishing and liftability to characteristic $0$ for Picard rank one del Pezzo surfaces). Thus one can wonder if it is possible to produce non-normal plt centers on some cones over these surfaces, following the method of \cite{Cascini_Tanaka_Plt_threefolds_with_non_normal_centres_in_char_2}. 
Unfortunately, a sole instance of Kawamata--Viewheg vanishing failure is not sufficient to produce such an example. We were unable to find the right configuration of ample divisor and boundary for the del Pezzo surface in characteristic $5$ described in \cite{ABL}: actually, our calculations led us to believe that no example of non-normal plt center can be produced by taking an appropriate cone over that surface.

For the characteristic $3$ del Pezzo surface constructed in \cite{Bernasconi_Kawamata-Viehweg_vanishing_fails}, we find a configuration leading to a non-normal cone boundary. However the cone pair we obtain has a minimal lc center at the vertex, so we do not get a plt pair. Nonetheless we can show that the cone divisor is not seminormal at the vertex, contrarily to what would happen in characteristic $0$ (see \cite[4.20.(3)]{Kollar_Singularities_of_the_minimal_model_program}).  

\begin{theorem_intro}\label{intro:example}[\autoref{section:counterexample}]
There exists a threefold lc pair $(X,E)$ with reduced boundary over an algebraically closed field of characteristic $3$, such that $E$ is not $S_2$ nor seminormal.
\end{theorem_intro}
In the example we construct, the reduced boundary $E$ contains every log canonical center of the pair $(X,E)$ (see \autoref{corollary:counterex}). So this shows that \autoref{theorem_intro:semi_normality} can fail in characteristic $3$.

In the last section of this note (\autoref{section:sufficient_conditions}), we give two sets of sufficient conditions on a Picard rank one variety $X$ (of any dimension) for the existence of a non-normal cone boundary over $X$. We believe these conditions to be of independent interest. Indeed, the calculations of \cite{Cascini_Tanaka_Plt_threefolds_with_non_normal_centres_in_char_2} are quite involved, rely on a precise local analysis of the relative cone, and require the vanishing of many cohomology groups. In particular, it is not immediate to adapt their method to a potential new example. By contrast, our conditions are short and only concern the global geometry of $X$: in one case a single failure of Kawamata--Viehweg vanishing for a nef divisor, in the other a failure of Kawamata--Viehweg for an anti-ample divisor and the existence of a boundary with small numerical class. Our counterexample in characteristic $3$ is obtained by using one of these conditions.

Let us also mention that examples of pathological behaviour of threefold singularities in positive characteristic are often constructed by considering cones constructed from $\mathbb{Q
}$-Cartier polarizations (see \cite{Cascini_Tanaka_Plt_threefolds_with_non_normal_centres_in_char_2,Bernasconi_Kawamata-Viehweg_vanishing_fails, ABL}). The birational singularities of cones with respect to Cartier divisors are well-understood (see for example \cite[3.1]{Kollar_Singularities_of_the_minimal_model_program}). However, to the best of the authors' knowledge the case of $\bQ$-Cartier polarizations in arbitrary characteristic has not received a similar treatment in the literature, although this construction is well-known to the experts. For the convenience of the reader we gather in \autoref{section:birational_sing_cones} some foundational results about singularities of cones, and more precise statements about three-dimensional cones are given in \autoref{section:birational_sing_threefold_cones}.


\bigskip
Lastly, as far as threefolds in positive characteristic are concerned, our results leave us with the following interrogations:
\begin{question}
Let $(X,\Delta)$ be a threefold pair over a perfect field of characteristic $p>0$.
    \begin{enumerate}
        \item If $p=5$, are plt centers normal? What about minimal lc centers?
        \item If $p\geq 5$, are unions of lc centers seminormal? or even weakly normal? What about union of slc centers?
        \item Are Cartier boundaries demi--normal in characteristics $\leq 5$? I.e., does \autoref{prop:demi_normality} fail in characteristic $p\leq 5$?
    \end{enumerate}
\end{question}


\subsection{Acknowledgments}
We thank Christopher Hacon and Karl Schwede for useful discussions related to this work. 
EA was supported by SNF\#P500PT 203083, QP was supported by the Koll\'{a}r fund at the University of Utah.

\section{Preliminaries} 

\subsection{Conventions and notations}
We will work with schemes that are not necessarily defined over a field, but we keep most of the usual terminology. 

Throughout this article we work over an affine base $B$ which is excellent, of pure finite dimension, and admits a dualizing complex. Unless stated otherwise, all the schemes appearing in this note are essentially of finite type over $B$.

By a \emph{variety} we shall mean a separated reduced equidimensional excellent scheme over $B$. Note that a variety in our sense might be reducible, even disconnected. A \emph{curve} (resp. a \emph{surface}, resp. a \emph{threefold}) is a variety of absolute dimension one (resp. two, resp. three). In particular the dimension shall always mean the absolute dimension and not the relative dimension over $B$.

Let $X$ be a variety and $\sF$ a coherent $\sO_X$-module. We say that $\sF$ is $S_i$ if $\depth_{\sO_{X,x}}\sF_x\geq \min\{i,\dim\sF_x\}$ for every $x\in X$.

If $X$ is a reduced Noetherian scheme, its \emph{normalization} is defined to be its relative normalization along the structural morphism $\bigsqcup_\eta\Spec(k(\eta))\to X$ where $\eta$ runs through the generic points of $X$. Recall that $X$ is normal if and only if it is regular in codimension one and $\sO_X$ is $S_2$.

If $X$ is a variety, we define $\bQ$-divisors and $\bQ$-Cartier divisors the usual way. We usually use these notions when $X$ is (demi-)normal.

A scheme $X$ essentially of finite type over $B$ has a dualizing complex $\omega_X^\bullet$ and a dualizing sheaf $\omega_X:=h^{-i}\left(\omega_X^\bullet\right)$ where $i:=\max\{j\mid h^{-j}\left(\omega_X^\bullet\right)\neq 0\}$. See \cite[\S 2.1]{Bhatt&Co_MMP_for_3folds_in_mixed_char} for more details. If $X$ is a variety and $\omega_X$ is invertible in codimension one, it defines a canonical divisor $K_X$.

A \textit{pair} is the data $(X,\Delta)$ where $X$ is a demi-normal variety, $\Delta$ is a $\bQ$-divisor with coefficients in $[0,1]$ such that no component of $\Delta$ is contained in $\Sing(X)$, and satisfying that $K_X+\Delta$ is $\bQ$-Cartier. If $f\colon Y\to X$ is a proper birational morphism from a normal variety, then there is a $\bQ$-divisor $\Delta_Y$ on $Y$ such that 
		$$K_Y+\Delta_Y=f^*(K_X+\Delta)$$
and $\Delta_Y$ is uniquely defined if we assume $f_*K_Y=K_X$, which we always do. Running through all such $f\colon Y\to X$, we define the singularity of the pair $(X,\Delta)$ (e.g. \textit{lc}, \textit{klt}) the usual way as in \cite{Kollar_Mori_Birational_geometry_of_algebraic_varieties}. See \cite[\S 2.5]{Bhatt&Co_MMP_for_3folds_in_mixed_char} for details.

Let $(X,\Delta)$ be a pair. We say that a reduced closed subscheme $Z\subset X$ is an \emph{lc center} of $(X,\Delta)$ if there exists a proper birational morphism $f\colon Y\to X$ and an $f$-exceptional prime divisor $E\subset Y$ such that $\coeff_{E}\Delta_Y=1$ and $f(E)=Z$.

We refer to \cite[\S 2.5]{Bhatt&Co_MMP_for_3folds_in_mixed_char} for the various notions of positivity (ampleness, nefness, etc.) for $\bQ$-Cartier $\bQ$-divisors on proper $B$-schemes.

\subsection{Universal homeomorphisms}
Recall that a morphism $f\colon X\to Y$ is a \textbf{universal homeomorphism} if for every $T\to Y$, the base-change $f_T\colon X_T\to T$ is an homeomorphism. Extensions of rings inducing universal homeomorphisms have the following algebraic property:

\begin{proposition}[{\cite[0CNE]{StacksProject}}]\label{prop:CA_description_of_univ_homeo}
Let $A\subset B$ be an extension of rings. Then the following are equivalent:
    \begin{enumerate}
        \item the induced morphism $\Spec B\to \Spec A$ is a universal homeomorphism;
        \item every finite subset of $B$ is contained in a sub-extension $A[b_1,\dots,b_n]\subset B$ such that for every $i=1,\dots,n$ we have either:
            \begin{itemize}
                \item $b_i^2,b_i^3\in A[b_1,\dots,b_{i-1}]$, or
                \item there exists a prime $p>0$ such that $b_i^p,pb_i\in A[b_1,\dots,b_{i-1}]$.
            \end{itemize}
    \end{enumerate}
\end{proposition}

Next we recall some sufficient conditions on a locally Noetherian scheme $X$ for its normalization morphism to be a universal homeomorphism. 

\begin{definition}[{\cite[0BPZ]{StacksProject}}]
A local ring $\sO$ is called \textbf{geometrically unibranch} if $\sO_\text{red}$ is a domain, its integral closure $\sO'$ in $\Frac(\sO_\text{red})$ is local, and the extension of residue fields $k(\sO)\subset k(\sO')$ is purely inseparable.

A scheme $X$ is called geometrically unibranch if every local ring $\sO_{X,x}$ is geometrically unibranch.
\end{definition}

\begin{lemma}[{\cite[0BQ4]{StacksProject}}]\label{lemma:irred_etale_neigh_implies_geom_unibranch}
Let $X$ be a scheme, and $x\in X$ a point. Assume that: if $(x'\in X')\to (x\in X)$ is any \'{e}tale neighbourhood, then $X'$ is irreducible at $x'$. Then $\sO_{X,x}$ is geometrically unibranch.
\end{lemma}

\begin{lemma}\label{lemma:geom_unibranch_implies_univ_homeo_normalization}
Let $X$ be an excellent geometrically unibranch scheme. Then the normalization morphism $X^\nu\to X$ is a universal homeomorphism.
\end{lemma}
The excellent hypothesis implies that of local Noetherianity, and in this context the normalization is defined: see \cite[035E]{StacksProject}.
\begin{proof}
By \cite[0C1S]{StacksProject} the normalization $\nu\colon X^\nu\to X$ is universally bijective. Since $\nu$ is surjective \cite[035Q]{StacksProject}, and moreover universally closed by the excellence hypothesis, we obtain that $\nu$ is a universal homeomorphism.
\end{proof}

\subsection{Weak normality and semi-normality}
We now recall the definition of semi-normal and weakly normal schemes. These notions originally appeared in \cite{Andreotti_Norguet_Convexite_holomorphe, Andreotti_Bombieri_On_the_homeomorphisms_of_varieties, Traverso_Seminormality_and_Picard_group}. We follow the more recent treatment given in \cite[ I.7]{Kollar_Rational_curves}.

\begin{definition}[{\cite[I.7.2.1]{Kollar_Rational_curves}}]\label{def:wk_and_semi_normality}
Let $X$ be an excellent reduced scheme with normalization $\nu\colon X^\nu\to X$. A scheme $X'$ together with a morphism $\varphi\colon X'\to X$ is called the \textbf{weak normalization} (resp. the \textbf{semi-normalization}) of $X$ if the following properties hold:
    \begin{enumerate}
        \item the normalization $X^\nu\to X$ factors through $\varphi$,
        \item $\varphi$ is finite surjective, and an isomorphism over the generic points of $X$, 
        \item $\varphi$ is an homeomorphism, and for any $x\in X$ with unique preimage $x'\in X'$, the field extension $k(x)\subset k(x')$ is purely inseparable (resp. an isomorphism),
        \item if $\varphi'\colon Y\to X$ also satisfies the above properties, then there is a factorization
                $$\varphi\colon X'\longrightarrow Y\overset{\varphi'}{\longrightarrow} X.$$
    \end{enumerate}
We say that $X$ is \textbf{weakly normal} (resp. \textbf{semi-normal}) if $\varphi\colon X'\to X$ is the identity morphism.
\end{definition}

Let us also recall the notion of demi-normality: it will appear in \autoref{section:demi_normality}.
\begin{definition}
A reduced Noetherian scheme $X$ is \textbf{demi-normal} if it is $S_2$ and for every codimension one point $\eta\in X$, the local ring $\sO_{X,\eta}$ is either regular or a nodal singularity (see \cite[1.41]{Kollar_Singularities_of_the_minimal_model_program}).
\end{definition}

See \cite[\S 2.3.1]{Posva_Gluing_for_surfaces_in_mixed_char} for the relations between demi-normality, semi-normality and weak normality.

\begin{lemma}[{\cite[I.7.2.3]{Kollar_Rational_curves}}]\label{lemma:existence_of_wk_normalizations}
Let $X$ be a reduced excellent scheme. Then $X$ has a weak normalization and a semi-normalization. They coincide if $X$ is a $\bQ$-scheme.
\end{lemma}

An example of weakly normal schemes that will be relevant in this note is given by the following lemma.
\begin{lemma}\label{lemma:divisor_is_weakly_normal}
Let $(X,\Delta+E)$ be an lc pair, where the residue fields of $X$ have characteristics $\neq 2$, and where $E$ is a (possibly reducible) divisor with coefficients $1$. Assume that $E$ is $S_2$. Then $E$ is weakly normal.
\end{lemma}
\begin{proof}
By the classification of lc singularities on surfaces \cite[2.31]{Kollar_Singularities_of_the_minimal_model_program}, we see that $E$ has at worse nodal singularities in codimension $1$. Since $E$ is also $S_2$, it is demi-normal. Now demi-normal schemes in char $\neq 2$ are weakly normal \cite[2.3.8]{Posva_Gluing_for_surfaces_in_mixed_char}. 
\end{proof}

Let us next give an algebraic description of weak normality. We stick to reduced excellent rings for simplicity.

\begin{proposition}[{\cite[\S 1]{Yanagihara_Weakly_normal_ring_extensions}}]\label{prop:CA_characterization_of_weak_normality}
Let $A$ be a reduced excellent ring with normalization $\bar{A}$.
    \begin{enumerate}
        \item $A$ is semi-normal if and only if: for every $x\in\bar{A}$, the fact that $x^2,x^3\in A$ implies that $x\in A$.
        \item $A$ is weakly normal if and only if: $A$ is semi-normal, and for every $x\in \bar{A}$ the fact that $x^p,px\in A$ for some prime $p>0$ implies that $x\in A$.
    \end{enumerate}
\end{proposition}

Notice the similarity between the algebraic conditions for semi-normality and weak normality given above, and the algebraic description of universal homeomorphisms given in \autoref{prop:CA_description_of_univ_homeo}. We take advantage of this similarity in order to obtain the following factorization result:

\begin{proposition}\label{prop:factorization_in_weakly_normal_case}
Let $f\colon Y\to X$ be a morphism of excellent reduced schemes. Assume that:
    \begin{enumerate}
        \item each irreducible component of $Y$ dominates a component of $X$, and
        \item $Y$ is weakly normal.
    \end{enumerate}
Then $f$ factors through the weak normalization $w\colon X^\text{wn}\to X$.
\end{proposition}
\begin{proof}
To give a morphism $g\colon Y\to X^\text{wn}$ means giving a continuous map $|g|\colon |Y|\to |X^\text{wn}|$ of topological spaces, together with a map of sheaves $g^\sharp\colon \sO_{X^\text{wn}}\to g_*\sO_Y$. By definition $X^\text{wn}\to X$ is an homeomorphism, so we take $|g|=|f|$, and to conclude we need to show that the map $\sO_X\hookrightarrow f_*\sO_Y$ admits a factorization $\sO_X\to w_*\sO_{X^\text{wn}}\to f_*\sO_Y$.

Let $Y^\nu\to Y$ be the normalization of $Y$ and $\nu_X\colon X^\nu\to X$ the normalization of $X$. By \cite[035Q.(4)]{StacksProject} there is a commutative diagram
        $$\begin{tikzcd}
        Y^\nu \arrow[rr]\arrow[d, "f^\nu"] && Y\arrow[d, "f"] \\
        X^\nu\arrow[r] & X^\text{wn}\arrow[r, "w"]&  X
        \end{tikzcd}$$
which induces the commutative diagram of $\sO_X$-algebras
        $$\begin{tikzcd}
        (\nu_X)_*f^\nu_*\sO_{Y^\nu} && f_*\sO_Y\arrow[ll, hook] \\
        (\nu_X)_*\sO_{X^\nu}\arrow[u, hook] & w_*\sO_{X^\text{wn}}\arrow[l, hook]\arrow[ul, hook, "\pi"]& \sO_X\arrow[u, hook, "\pi|_{\sO_X}"]\arrow[l, hook].
        \end{tikzcd}$$

We claim that the arrow $\pi\colon w_*\sO_{X^\text{wn}}\to (\nu_X)_*f^\nu_*\sO_{Y^\nu}$ factors through $f_*\sO_Y$. Indeed, 
take $s\in w_*\sO_{X^\text{wn}}(U)$ for some affine open set $U$ of $X$. By \autoref{prop:CA_description_of_univ_homeo} we can find $t_1,\dots, t_n\in \sO_X(U)$ such that $s\in \sO_X(U)[t_1,\dots,t_n]$ where for every $i=1,\dots, n$ we have either
    \begin{enumerate}
        \item $t_i^2,t_i^3\in \sO_X(U)[t_1,\dots,t_{i-1}]$, or
        \item $t_i^p,pt_i\in \sO_X(U)[t_1,\dots,t_{i-1}]$ for some prime $p>0$.
    \end{enumerate}
Let $V=f^{-1}(U)$ and $V^\nu$ be its preimage in $Y^\nu$. Then we can say that $\pi(s)\in \sO_{Y^\nu}(V^\nu)$ belongs to $\sO_Y(V)[\pi(t_1),\dots \pi(t_n)]$ where for every $i=1,\dots, n$ we have either
    \begin{enumerate}
        \item $\pi(t_i)^2,\pi(t_i)^3\in \sO_Y(V)[\pi(t_1),\dots,\pi(t_{i-1})]$, or
        \item $\pi(t_i)^p,p\pi(t_i)\in \sO_Y(V)[\pi(t_1),\dots,\pi(t_{i-1})]$ for some prime $p>0$.
    \end{enumerate}
Since $Y$ is weakly normal, it follows from \autoref{prop:CA_characterization_of_weak_normality} (applied on an affine cover of $V$) that $\pi(t_1)\in \sO_Y(V)$. In turn, it follows that $\pi(t_2)\in \sO_Y(V)$, and by induction we arrive at the conclusion that all the $\pi(t_i)$'s belong to $\sO_Y(V)$. Therefore $\pi(s)$ belongs to $\sO_Y(V)=f_*\sO_Y(U)$. As $s$ was arbitrary, our claim is proved, and we are provided with the morphism $Y\to X^\text{wn}$ that we were looking for.
\end{proof}

\begin{corollary}\label{corollary:factorization_in_normal_case}
In the situation of \autoref{prop:factorization_in_weakly_normal_case}, if the normalization $X^\nu\to X$ is a universal homeomorphism, then $Y\to X$ factorizes through $X^\nu$.
\end{corollary}
\begin{proof}
Indeed, if $X^\nu\to X$ is a universal homeomorphism then $X^\nu$ is equal to the weak normalization of $X$.
\end{proof}

\begin{remark}
One can see \autoref{prop:factorization_in_weakly_normal_case} as a universal factorization property of the weak normalization, similar to that of semi-normalization (see \cite[I.7.2.3.3]{Kollar_Rational_curves}). However, one cannot drop the assumption that every component of $Y$ dominates a component of $X$. The following counter-example is given in \cite[I.7.2.2,I.7.2.3]{Kollar_Rational_curves}. Let $k$ be a field of characteristic $p>0$ and let $R_n=k[x^{p^n}, yx^i:\ 0\leq i\leq p^n-1]$. The weak normalization of $R_n$ is $k[x,y]$. If we quotient $R_n$ by the ideal $(yx^i: \ 0\leq i\leq p^n-1)$, we obtain the (weakly) normal ring $k[x^{p^n}]$. However there is no morphism $\alpha\colon \Spec k[x^{p^n}]\to \Spec k[x,y]$ such that the diagram
        $$\begin{tikzcd}
        & \Spec k[x^{p^n}]\arrow[d, hook]\arrow[dl, "\alpha"] \\
        \Spec k[x,y]\arrow[r] & \Spec R_n
        \end{tikzcd}$$
is commutative. (Indeed there is no valid choice for the image $\alpha^*(x)\in k[x^{p^n}]$.) 
\end{remark}

\subsection{Birational vanishing theorems for threefolds}

In this section we collect the vanishing theorems for a dlt modification of a log canonical threefold that we will need in the proof of \autoref{theorem_intro:main}  and \autoref{theorem_intro:semi_normality}. 

\begin{proposition}\label{mainvanish}
\cite[Prop. 14]{Arv22} Let $(X,\Delta)$ be a quasi--projective 
log canonical threefold without zero-dimensional log canonical centers and with closed points of perfect residue field of characteristics $p>5$. 
There exists a proper birational morphism $f\colon (X^\text{dlt}, \Delta_{X^\text{dlt}}+E)\to  (X,\Delta)$ 
 such that $E=\Exc(f)$ is the reduced exceptional divisor of $f$, the pair $(X^\text{dlt}, \Delta_{X^\text{dlt}}+E)$ is dlt, $f^*(K_X+\Delta)= K_X + \Delta_{X^\text{dlt}}+E$ and $R^1f_*\sO_X(-E)=0$.
\end{proposition}

The above proposition will allow us to say something about \emph{minimal} log canonical centers. On the contrary, the proposition below (which takes the main result of \cite{Kawakami_KVV_for_log_CY_surfaces_in_large_char} as an essential input) will be applied to the study of the union of all log canonical centers for large residue characteristics. 

\begin{proposition}\cite[Prop. 15]{Arv22}\label{MainvanII} There exists an integer $p_0$ with the following property. Assume that $\Delta\geq 0$ is a boundary with standard coefficients. Let $(X,\Delta)$ be a quasi--projective log canonical threefold  with closed points of perfect residue fields, each of characteristic $p>p_0$.  Then there exists a proper birational morphism $f\colon (X^\text{dlt}, \Delta_{X^\text{dlt}}+E)\to  (X,\Delta)$ 
 such that $E=\Exc(f)$ is the reduced exceptional divisor of $f$, the pair $(X^\text{dlt}, \Delta_{X^\text{dlt}}+E)$ is dlt, $f^*(K_X+\Delta)= K_X + \Delta_{X^\text{dlt}}+E$ and $R^1f_*\sO_X(-E)=0$.
 \end{proposition}

With the factorization result \autoref{corollary:factorization_in_normal_case}, the vanishing theorems \autoref{mainvanish} and \autoref{MainvanII} at hand, we are ready to prove \autoref{theorem_intro:main} and \autoref{theorem_intro:semi_normality}. We do so in \autoref{section:normality}. 

\subsection{Some generalities on cones}
For the convenience of the reader, we recall some basic facts about cones. The results of this section will be used in the construction of a non semi-normal log canonical center \autoref{section:counterexample}, as well as in \autoref{section:sufficient_conditions}. Cone constructions for Cartier polarizations are reviewed in \cite[\S 3.1]{Kollar_Singularities_of_the_minimal_model_program}. In addition to what is presented below, see also \cite[\S 2.3]{Bernasconi_Kawamata-Viehweg_vanishing_fails} for the $\bQ$-Cartier case.

Let $k$ be an arbitrary field and let $X$ be a normal connected excellent proper $k$-scheme. Let $A$ be a $\bQ$-Cartier $\bZ$-divisor on $X$ that is ample. We define the \textbf{absolute affine cone} of $X$ with respect to $A$ to be
        $$\aC(X,A)=\Spec_k \bigoplus_{d\geq 0}H^0(X,\sO(dA)).$$
We define the \textbf{relative affine cone} of $X$ with respect to $A$ to be
        $$\aBC(X,A)=\Spec_X \bigoplus_{d\geq 0}\sO(dA)$$
where the $\sO_X$-algebra structure on $\bigoplus_{d}\sO(dA)$ is given by the reflexified tensor product. We have a diagram
        \begin{equation}\label{diagram:cones}
        \begin{tikzcd}
        \aBC(X,A)\arrow[r, "f"]\arrow[d, "\rho"] & \aC(X,A) \\
        X &
        \end{tikzcd}
        \end{equation}
The absolute cone $\aC(X,A)$ has a closed point $v$, called its \textbf{vertex} defined by the irrelevant ideal $\bigoplus_{d>0}H^0(X,\sO(dA))$. The morphism $\rho$ has a section $X^-$, defined by the ideal $\bigoplus_{d>0}\sO(dA)$ of $\sO_{\aBC(X,A)}$. The following claim is standard.

\begin{claim}\label{propertiesofcone}
With the notations as above:
    \begin{enumerate}
        \item $\rho$ is of finite type of relative dimension one, and over a big open subset of $X$ is it an $\bA^1$-bundle;
        \item $\aBC(X,A)$ is normal and $X^-$ is $\bQ$-Cartier;
        \item $\aC(X,A)$ is affine, connected, normal and of finite type over $k$;
        \item the morphism $f$ induces an isomorphism 
    $\aBC(X,A)\setminus X^-\cong \aC(X,a)\setminus \{v\}$.
        \item $\aBC(X,A)\setminus X^-\cong \Spec_X\bigoplus_{d\in\bZ}\sO(dA)$.
    \end{enumerate}
\end{claim}

\noindent Let $E\subset X$ be a prime divisor (not necessarily $\bQ$-Cartier). We define a prime divisor $E_{\aC(X,A)}\subset \aC(X,A)$ by $E_{\aC(X,A)}=f_*\rho^{-1}E$. We also define a prime divisor $E_{\aBC(X,A)}\subset \aBC(X,A)$ by $E_{\aBC(X,A)}=\red(\rho^{-1}E)$. Then
    \begin{equation*}
    \sI_{E_{\aBC(X,A)}}=\bigoplus_{d\geq 0}\sO_X(dA-E)\hookrightarrow \bigoplus_{d\geq 0}\sO_X(dA)=\mathcal{O}_{\aBC(X,A)}
    \end{equation*}
is the (graded) ideal defining $E_{\aBC(X,A)}$ inside $\aBC(X,A)$.

Consider for all $d\in \bZ$ the exact sequence of $\mathcal{O}_X$-modules
        \begin{equation}\label{eqn:restrictrion_exact_seq}
            0\to \sO_X(dA-E)\to \sO_X(dA)\to\sE_d\to 0,
        \end{equation}
where $\sE_d$ is defined by the exactness of the sequence. Then $\sE_d$ is a rank one $S_1$-sheaf supported on $E$ \cite[Lemma 2.60]{Kollar_Singularities_of_the_minimal_model_program}, and hence a torsion free $\sO_E$-module \cite[Proposition 2.6]{Schwede2008GENERALIZEDDA}. 

\begin{claim}\label{claim:algebra_structure_on_restriction}
The direct sum $\bigoplus_{d\in\bZ}\sE_d$ has a structure of $\bZ$-graded $\sO_E$-module.
\end{claim}
\begin{proof}
This direct sum is the quotient of the $\sO_X$-module $\sA=\bigoplus_{d\in\bZ} \sO_X(dA)$ by the graded sub-module $\sI=\bigoplus_{d\in\bZ}\sO_X(dA-E)$. Since $A$ is a $\bZ$-divisor, the module $\sA$ obtains a graded $\sO_X$-algebra structure via the pairing
        $$\sO(dA)\otimes \sO(eA)\to \sO(dA)[\otimes]\sO(eA)=\sO((d+e)A).$$
So we only have to show that $\sI$ is a graded ideal of $\sA$. We observe that the product structure takes $\sO(dA-E)\otimes\sO(eA)$ to $\sO((d+e)A-E)$, so $\sI$ is indeed an ideal.
\end{proof}


       
       Similarly, $$I_{E_{\aC(X,A)}}=\bigoplus_{d\geq 0}H^0(X,\sO(dA-E))$$ is the (graded) ideal defining $E_{\aC(X,A)}$ inside $\aC(X,A)$
and the cokernel $\sO_{E_{\aC(X,A)}}$ is by definition the $k$-algebra $\bigoplus_{d\geq 0}\im\psi_d$, where    
    $$\psi_d\colon H^0(X,\sO(dA))\longrightarrow H^0(E,\sE_d)$$ 
is the map on cohomology induced by \autoref{eqn:restrictrion_exact_seq}.

\begin{proposition}\label{prop:S_2_at_vertex}
 $E_{\aC(X,A)}$ is $S_2$ at the vertex if and only the map $H^1(X,\sO(dA-E))\to H^1(X,\sO(dA))$ induced by \autoref{eqn:restrictrion_exact_seq} is injective for every $d\geq 0$.
\end{proposition}
\begin{proof}
Let $U$ be the complement of the vertex $v\in\aC(X,A)$. Applying \cite[III.exerc.2.3]{Hartshorne_Algebraic_Geometry} to $\sO_{E_{\aC(X,L)}}=\sO$, we find the exact sequence
    $$H^0\left(E_{\aC(X,L)},\sO\right)\to H^0\left(E_{\aC(X,L)}\cap U, \sO\right)\to H^1_v\left(E_{\aC(X,L)},\sO\right)\to H^1\left(E_{\aC(X,L)},\sO\right)=0.$$
So $E_{\aC(X,L)}$ is $S_2$ if and only if $H^0(E_{\aC(X,L)},\sO)\to H^0(E_{\aC(X,L)}\cap U, \sO)$ is surjective. Now $H^0(E_{\aC(X,L)},\sO)=\bigoplus_{d\geq 0}\im\psi_d$ as we have seen above. The isomorphism $U\cong \Spec_X\bigoplus_{d\in\bZ}\sO(dA)$ induces an isomorphism $E_{\aC(X,A)}\cap U\cong \Spec_X \bigoplus_{d\in\bZ}\sE_d$. So the restriction $\pi$ of $\rho$ to $E_{\aC(X,A)}\cap U$ is affine, and 
we obtain
        $$H^0(E_{\aC(X,L)}\cap U,\sO)
        =H^0(X,\pi_*\sO_{E_{\aBC}\cap U})
        = H^0(X,\bigoplus_{d\in\bZ}\sE_d)
        =\bigoplus_{d\in\bZ}H^0(E,\sE_d),
        $$
where the last equality comes from that the sheaves $\sE_d$ are supported on $E$.
        
Therefore, the morphism $H^0(E_{\aC(X,L)},\sO)\to H^0(E_{\aC(X,L)}\cap U, \sO)$ is the direct sum 
        $$\bigoplus_{d\in\bZ} \left[\im\psi_d\hookrightarrow H^0(E,\sE_d)\right].$$
In other words, $E_{\aC(X,A)}$ is $S_2$ at the vertex if and only if $\psi_d$ is surjective for every $d\in \bZ$. Since we have the exact sequences
        $$H^0(X,\sO(dA))\overset{\psi_d}{\longrightarrow}H^0(E,\sE_d)
        \to H^1(X,\sO(dA-E))\overset{\alpha_d}{\longrightarrow} H^1(X,\sO(dA))$$
induced by \autoref{eqn:restrictrion_exact_seq},
we see that $\psi_d$ is surjective if and only if $\alpha_d$ is injective. 

Finally, we claim that the group $H^0(E,\sE_d)$ is trivial for $d<0$, which implies that $\alpha_d$ is injective for every $d< 0$. This will concludes the proof.

Fix $d<0$. For $n>0$ divisible enough, the sheaf $\sE_{nd}$ is the restriction of an invertible sheaf whose dual is very ample. Thus $\sE_{nd}^{-1}$ is ample, and so $H^0(E,\sE_{nd})=0$. Now consider the $n$-fold product
        $$\sE_{d}\longrightarrow\sE_{nd},\quad s\mapsto s^{\bullet n}$$
induced by the structure of graded algebra of $\bigoplus_{d\in\bZ}\sE_d$ (see \autoref{claim:algebra_structure_on_restriction}). Since generically along $E$ the sheaf $\sO_X(A)$ is invertible, this $n$-fold product is generically an isomorphism. Since the $\sE_i$'s are torsion-free, this implies that the induced product on global sections
        $$H^0(E,\sE_{d})\longrightarrow H^0(E,\sE_{nd}),\quad s\mapsto s^{\bullet n}$$
is injective. But $H^0(E,\sE_{nd})$ is trivial, so $H^0(E,\sE_d)$ is also trivial.
\end{proof}

\begin{lemma}\label{lemma:psi_eventually_surjects}
$\psi_d$ is surjective for $d\gg 0$.
\end{lemma}
\begin{proof}
Say that $m$ is the Cartier index of $A$. Then $\sE_{d+m}=\sE_d\otimes\sO_E((mA)|_E)$. For every $i=0,\dots,m-1$, by Serre vanishing for the ample Cartier divisor $mA$ there exists $n_i>0$ such that $H^1(T,\sO((nm+i)A-E))=0$ for every $n\geq n_i$. If $N=\max\{n_i\}$ this implies that
        $$\psi_{nm+i}\colon H^0(T,\sO((nm+i)A))\to H^0(E,\sE_{nm+i}) \text{ is surjective }\forall n\geq N, i=0,\dots,m-1.$$
In other words, $\im\psi_d=H^0(E,\sE_d)$ for $d\geq Nm$.
\end{proof}

\begin{corollary}\label{corollary:partial_normalization_cone}
The morphism $\Spec_k \bigoplus_{d\geq 0}H^0(E,\sE_d)\to E_{\aC(X,A)}$ is a finite birational homeomorphism inducing equalities as extension of residue fields.
\end{corollary}
\begin{proof}
Consider the graded rings $R=\bigoplus_{d\geq 0} \im\psi_d$ and $R'=\bigoplus_{d\geq 0}H^0(E,\sE_d)$. Denote by $v'\in \Spec R'$ the closed point given by the irrelevant ideal. We have an inclusion $R\hookrightarrow R'$ preserving the degrees. By \autoref{lemma:psi_eventually_surjects} there exists $N\gg 0$ such that $R_d=R'_d$ for $d\geq N$. This shows that $\Spec R'\to \Spec R=E_{\aC(X,A)}$ is  finite.

Take $f\in R_d=R'_d$ for $d\geq N$. Then for any $s\in R'_e$ we have $fs\in R_{e+d}=R_{e+d}'$, so $s=fs/f\in R[f^{-1}]$. This shows that $R[f^{-1}]=R'[f^{-1}]$. Since $\{\Spec R[f^{-1}]\mid f\in R_d, d\geq N\}$ is an open cover of $\Spec R\setminus\{v\}$, and similarly $\{\Spec R'[f^{-1}]\mid f\in R_d, d\geq N\}$ is an open cover of $\Spec R'\setminus\{v'\}$, we obtain that $\Spec R'\setminus \{v'\}\to \Spec R\setminus\{v\}$ is an isomorphism.

To conclude we only have to show that $k(v')=k(v)$. We have
        $$\Spec R'\times_{\Spec R} k(v)=\Spec_k \left(H^0(E,\sO_E)\oplus \bigoplus_{d=1}^{N-1}H^0(E,\sE_d)/\im\psi_d\right).$$
Thus 
      $$\Spec k(v')=\red\left(\Spec R'\times_{\Spec R} k(v)\right)=\Spec_kH^0(E,\sO_E)=\Spec k(v),$$
which concludes the proof.
\end{proof}

\subsubsection{Birational singularities of cones}\label{section:birational_sing_cones}

Consider anew a proper pair $(X,E)$ such that $K_X+E\sim_{\mathbb{Q}}rA$ where $A$ is an ample $\mathbb{Q}$-Cartier $\mathbb{Z}$-divisor and $r\in \mathbb{Q}$. The assumption $K_X+E\sim_{\mathbb{Q}}rA$ assures that $K_{\aC(X,A)}+E_{\aC(X,A)}$ is $\mathbb{Q}$-Cartier \cite[2.4.3]{Bernasconi_Kawamata-Viehweg_vanishing_fails}. In this subsection we consider the birational singularities of the affine cone around its vertex. We have (see 
\cite[(2.4.2)]{Bernasconi_Kawamata-Viehweg_vanishing_fails}) that
        $$f^*\left(K_{\aC(X,A)}+E_{\aC(X,A)}\right)=
        K_{\aBC(X,A)}+E_{\aBC(X,A)}+(1+r)X^-,$$
where $f\colon \aBC(X,A) \to \aC(X,A)$ is as in \autoref{diagram:cones}, so it is equivalent to consider the birational singularities of the relative cone along its section.

\begin{definition}\label{inversion}We will say that \emph{inversion of adjunction holds for the pair $(\aBC(X,A),E_{\aBC(X,A)}+X^-)$ along $X^-$}, if the following implications hold: \begin{itemize}
\item If $(X^-,\Diff_{X^-}(E_{\aBC(X,A)}))$ is klt then $(\aBC(X,A),E_{\aBC(X,A)}+X^-)$ is plt around $X^-$, and
\item If $(X^-,\Diff_{X^-}(E_{\aBC(X,A)}))$ is lc then $(\aBC(X,A),E_{\aBC(X,A)}+X^-)$ is lc around $X^-$.
\end{itemize}
\end{definition}

\noindent Since $X^-\cong X$ we see that $X^-$ is normal. Moreover, note that under this isomorphism we have
    \begin{equation}\label{eqn:adjunction_on_section}
        K_{X^-}+\Diff_{X^-}(E_{\aBC(X,A)})= K_X +E.
    \end{equation}
Indeed, this equality holds over the Cartier locus of $A$ where the relative cone is an $\bA^1$-bundle, and since the preimage of that Cartier locus in $\aBC(X,A)$ contains every codimension one point of $X^-$.


The following proposition is standard.

\begin{proposition}\label{lc klt assuming inv adjunction}Assume that $(X,E)$ is a pair such that $K_X+E\sim_{\mathbb{Q}}rA$ where $A$ is an ample $\mathbb{Q}$-Cartier $\mathbb{Z}$-divisor and $r\in \mathbb{Q}$. Assume inversion of adjunction for the pair $(BC_a(X,A),E_{BC_a(X,A)}+X^-)$ along $X^-$. Then:\begin{enumerate}
    \item If $r<0$ and $(X,E)$ is klt, then $(C_a(X,A), E_{C_a(X,A)})$ is klt at the vertex.
    \item If $r\leq 0$ and $(X,E)$ is lc, then $(C_a(X,A), E_{C_a(X,A)})$ is lc at the vertex.
\end{enumerate}
\end{proposition}

\begin{proof}
Let $f\colon \aBC(X,A)\to \aC(X,A)$ be the natural morphism. An elementary computation shows that 
    \begin{equation}\label{eqn:pullback_between_cones}
    f^*(K_{C_a(X,A)}+E_{C_a(X,A)})=K_{BC_a(X,A)}+E_{BC_a(X,A)}+(1+r)X^-,
    \end{equation}
see for example \cite[Proof of 3.1]{Kollar_Singularities_of_the_minimal_model_program}.
It is therefore sufficient to prove that $(\aBC(X,A), E_{\aBC(X,A)}+(1+r)X^-)$ is klt, respectively lc. 

Assume that $(X,E)$ is klt and $r<0$. Then by \autoref{eqn:adjunction_on_section}, since inversion of adjunction holds by hypothesis, we obtain that $(\aBC(X,A), E_{\aBC(X,A)}+X^-)$ is plt in a neighbourhood of $X^-$. Since $r<0$ we get that $(\aBC(X,A), E_{\aBC(X,A)}+(1+r)X^-)$ is klt on the same neighbourhood.

If we only assume that $(X,E)$ is lc and $r\leq 0$, then the same argument shows that the pair $(\aBC(X,A), E_{\aBC(X,A)}+(1+r)X^-)$ is lc in a neighbourhood of $X^-$.
\end{proof}

\begin{remark}
If the index of $A$ is not divisible by the characteristic, then the conclusions of \autoref{lc klt assuming inv adjunction} hold \emph{without assuming inversion of adjunction}. Indeed, this can be deduced from the Cartier case \cite[3.1]{Kollar_Singularities_of_the_minimal_model_program} by means of \cite[2.43.1]{Kollar_Singularities_of_the_minimal_model_program}. 
\end{remark}

\subsubsection{Birational singularities of cones over surfaces.}\label{section:birational_sing_threefold_cones}

For cones over surfaces, inversion of adjunction in the form of \autoref{inversion} is available:

\begin{lemma}\label{remark:inv_adj_for_threefolds}
Let $(X,E)$ be a normal surface pair over a perfect field. Let $A$ be an ample $\bQ$-Cartier $\bZ$-divisor on $X$ such that $K_X+E\sim_{\bQ}rA$ with $r\in\bQ$. Then inversion of adjunction (see \autoref{inversion}) holds for the threefold pair $(\aBC(X,A),E_{\aBC(X,A)}+X^-)$ along $X^-$ .
\end{lemma}
\begin{proof}
Thanks to \cite[Corollary 1.4]{Hacon_Witaszek_On_the_relative_MMP_for_threefolds_in_low_char}, the proof of \cite[Theorem 6.2]{Hacon_Xu_On_the_3dim_MMP_in_pos_char} works in every characteristic and without any assumption on the coefficients of $E_{\aBC(X,A)}$ (see also \cite[Corollary 1.5]{Hacon_Witaszek_The_MMP_for_threefolds_in_char_5}). To conclude we need to show that if $(X,E)$ is dlt then $\aBC(X,A)$ is $\mathbb{Q}$-factorial. Indeed, as $X$ is a potentially klt surface it is $\mathbb{Q}$-factorial \cite[Corollary 4.11]{Tanaka_MMP_for_excellent_surfaces} and hence by \cite[Proof of proposition 2.4]{Bernasconi_Kawamata-Viehweg_vanishing_fails} the relative cone $\aBC(X,A)$ is $\mathbb{Q}$-factorial.
\end{proof}



\begin{corollary}\label{Sing 3fold Cones}
Assume that $(X,E)$ is a surface pair such that $K_X+E\sim_{\mathbb{Q}}rA$ where $A$ is an ample $\mathbb{Q}$-Cartier $\mathbb{Z}$-divisor and $r\in \mathbb{Q}$. Then;\begin{enumerate}
    \item If $r<0$ and $(X,E)$ is klt, then $(\aC(X,A), E_{\aC(X,A)})$ is klt everywhere (not only at the vertex).
    \item If $r\leq 0$ and $(X,E)$ is lc, then $(\aC(X,A), E_{\aC(X,A)})$ is lc everywhere (not only at the vertex).
\end{enumerate}
\end{corollary}
\begin{proof} 
First assume that $r<0$ and that $(X,E)$ is klt. To show that $(\aC(X,A), E_{\aC(X,A)})$ is klt everywhere, by the crepant equation \autoref{eqn:pullback_between_cones} it suffices to show that the relative cone pair $(\aBC(X,A), E_{\aBC(X,A)})$ is klt everywhere. For the rest of the proof, let us write
        $$\sV=\aBC(X,A)\setminus X^-, \quad 
        \rho|_{\sV}=\rho^0\colon \sV\to X$$
where $\rho\colon \aBC(X,A)\to X$ is the structural morphism.

It follows from \autoref{lc klt assuming inv adjunction} that the pair $(\aBC(X,A), E_{\aBC(X,A)})$ is klt along the section $X^-$, since by \autoref{remark:inv_adj_for_threefolds} inversion of adjunction holds for threefold pairs. Since log resolution for threefolds is available in positive and mixed characteristics \cite{Cossart_Piltant_Resolution_of_singularities_for_3_folds_in_pos_char_I, Cossart_Piltant_Resolution_of_3folds_II}, we can apply the arguments of \cite[2.31-32]{Kollar_Mori_Birational_geometry_of_algebraic_varieties} to deduce that the non-klt locus of $(\aBC(X,A), E_{\aBC(X,A)})$ is closed. Therefore we obtain that: if $U_\text{klt}\subset \aBC(X,A)$ is the klt locus of the pair $(\aBC(X,A), E_{\aBC(X,A)})$, then $U_\text{klt}$ is open and $X^-\subset U_\text{klt}$. Therefore, $U^0_\text{klt}=U_\text{klt}\cap \sV$ intersects non-trivially every fiber of the morphism $\rho^0$.

By \autoref{propertiesofcone} we have
        $$\sV \cong \Spec_X\bigoplus_{d\in\bZ} \sO(dA),$$
and the $\bZ$-grading induces a $\bG_m$-action on $\sV$ such that $\rho^0$ is invariant. If $D$ is any prime divisor supported on $\Supp(E)$, then the ideal of $D_{\aBC(X,A)}\cap \sV$ is the graded ideal $\bigoplus_{d\in\bZ}\sO_X(dA-D)$. In particular the $\bG_m$-action preserves $D_{\aBC(X,A)}\cap\sV$. Therefore $\bG_m$ acts on the log pair $(\sV, E_{\aBC(X,A)}\cap \sV)$. That action preserves the open subset $U^0_\text{klt}\subseteq \sV$, because isomorphisms of log pairs preserve klt singularities.

Since $\rho^0$ is $\bG_m$-invariant, the $\bG_m$-action preserves the fibers of $\rho^0$. Moreover the restricted action on any fiber of $\rho^0$ is transitive. Since $U^0_\text{klt}$ intersects every fiber of $\rho^0$ non-trivially, we deduce that $U^0_\text{klt}=\sV$. Therefore $U_\text{klt}=\aBC(X,A)$, and we have proved that $(\aBC(X,A),E_{\aBC(X,A)})$ is klt everywhere.

In the case where $r\leq 0$ and $(X,E)$ is lc, the proof is similar.
\end{proof}

\section{Normality in characteristic $>5$}\label{section:normality}

\subsection{Normality of minimal lc centers}
The following result follows essentially from \cite{Filipazzi_Waldron_Connectedness_principle_char>5}.
\begin{lemma}\label{lemma:minimal_lc_centers_are_normal_up_to_homeo}
Let $(X,\Delta)$ be a semi-log canonical threefold whose closed points have perfect residue fields of characteristic different from $2,3 $ and $5$. Then the minimal lc centers of $(X,\Delta)$ are normal up to universal homeomorphism.
\end{lemma}
\begin{proof}
By \autoref{lemma:geom_unibranch_implies_univ_homeo_normalization} it is sufficient to show that the minimal lc centers are geometrically unibranch. The case where $(X,\Delta)$ is lc follows from \cite[Corollary 1.7]{Filipazzi_Waldron_Connectedness_principle_char>5} and stability of lc singularities under \'{e}tale base-changes. 

In general, let $Z\subset X$ is a minimal lc center, let $z\in Z$ and consider an \'{e}tale neighborhood $(z'\in Z')\to (z\in Z)$. We have to show that $Z'$ is irreducible at $z'$. After shrinking it if necessary, we may assume that there is an \'{e}tale neighborhood $(z'\in X')\to (z\in X)$ such that $Z'=Z\times_XX'$ \cite[02GI, 02GU]{StacksProject}. As normalization commutes with \'{e}tale base-changes \cite[07TD]{StacksProject}, if $(\bar{X},\bar{D}+\bar{\Delta})$ is the normalization of $X$ then $Y=\bar{X}\times_XX'\to X'$ is the normalization morphism. We also obtain crepant divisors, defined by the following diagram:
        $$\begin{tikzcd}
        (Y,D_Y+\Delta_Y)\arrow[d, "\mu"]\arrow[r] & (\bar{X},\bar{D}+\bar{\Delta})\arrow[d] \\
        (X',\Delta')\arrow[r, "f"] & (X,\Delta)
        \end{tikzcd}$$
Since $\mu$ is crepant, the lc centers of $(X',\Delta')$ are exactly the images of the lc centers of $(Y,D_Y+\Delta_Y)$. By construction $Z'$ is an lc center of $(X',\Delta')$. Since intersections of lc centers on $Y$ are union of lc centers by \cite[Corollary 1.7]{Filipazzi_Waldron_Connectedness_principle_char>5}, we see that if $Z'$ is reducible then it is not a minimal lc center. But $f$ is crepant, so it means that $Z$ is not a minimal lc center of $(X,\Delta)$, which is a contradiction. Thus $Z'$ is irreducible, as desired.
\end{proof}

\begin{theorem}\label{maintheorem}
Let $(X,\Delta)$ be a log canonical threefold whose closed points have perfect residue fields of characteristic different from $2,3 $ and $5$. If $C\subset X$ is a minimal log canonical center then $C$ is normal.
\end{theorem}

\begin{proof}
The strategy of proof is similar to \cite[Theorem 3.6]{DH16}. Firstly, if $C$ is a surface, then $(X,\Delta)$ is plt around $C$ and by \cite[Theorem 19]{Bernasconi_Kollar_vanishing_theorems_for_threefolds_in_char_>5} the surface $C$ is normal. If $C$ is a point there is nothing to prove. So for the rest of the proof we may assume that $C$ is a curve. We may also localize at any closed point of $C$, so that we can assume that $(x\in X,\Delta)$ is local and that $C$ is the unique minimal log canonical center of $(X,\Delta)$. 

Let $g\colon (Z, \Delta_Z+E)\to (X,\Delta_X)$ be the dlt model described in \autoref{mainvanish}. We can assume that $g(E_i)=C$ for each component $E_i$ of $E$. Since $(Z, \Delta_Z+ E)$ is dlt \cite[Theorem 19]{Bernasconi_Kollar_vanishing_theorems_for_threefolds_in_char_>5} implies that $E$ is $S_2$. Therefore $E$ is weakly normal by \autoref{lemma:divisor_is_weakly_normal}. Consider the short exact sequence associated to $E$ inside $Z$:
$$\begin{tikzcd} 0\arrow{r}& \mathcal{O}_Z(-E)\arrow{r} &\mathcal{O}_Z\arrow{r} &\mathcal{O}_E\arrow{r} & 0.\end{tikzcd}$$
By \autoref{mainvanish} it holds that $R^1g_*\mathcal{O}_Z(-E)=0$ and hence, pushing forward along $g$, on $X$ there is a short exact sequence on
$$\begin{tikzcd} 0\arrow{r}& g_*\mathcal{O}_Z(-E)\arrow{r} &g_*\mathcal{O}_Z \arrow{r} &g_*\mathcal{O}_E\arrow{r} &0. \end{tikzcd}$$
 We have $g_*\mathcal{O}_Z=\mathcal{O}_X$ and $g_*\mathcal{O}_Z(-E)=\mathcal{I}_C$ is the ideal sheaf of $C$. Let $\nu\colon C^\nu\to C$ denote the normalization. By \autoref{lemma:minimal_lc_centers_are_normal_up_to_homeo} this normalization morphism is a universal homeomorphism. Since $E$ is weakly normal, \autoref{corollary:factorization_in_normal_case} implies that there exists a factorization $\mathcal{O}_C\hookrightarrow \nu_*\mathcal{O}_{C^\nu}\hookrightarrow g_*\mathcal{O}_E$. Therefore we have the commutative diagram
 $$\begin{tikzcd}&&\mathcal{O}_C\arrow[hook]{rd}\\
 &&&\nu_*\mathcal{O}_{C^\nu}\arrow[hook]{d}\\
 \mathcal{O}_X\arrow[twoheadrightarrow]{rruu}\arrow[twoheadrightarrow]{rrr}&&&g_*\mathcal{O}_E.
 \end{tikzcd}$$
    Since $\sO_X\twoheadrightarrow g_*\sO_E$ is surjective, we see that $\mathcal{O}_C\hookrightarrow \nu_*\mathcal{O}_{C^\nu}$ is also a surjection. Therefore $C\cong C^\nu$, as claimed.
    \end{proof}


\subsection{Demi-normality of Cartier boundaries}\label{section:demi_normality}

In the Cartier case, the following result is due to \cite{ABP, Arv22}. We can build on it using cyclic covers to obtain a slight improvement.

\begin{proposition}\label{prop:demi_normality}
Let $(X,S+\Delta)$ be an slc threefold pair, where $S$ is an effective divisor. Assume that the residue fields at closed points are perfect of characteristics different from $2$, $3$ and $5$. If $S$ is $\bQ$-Cartier with index not divisible by any of the residue characteristics, then $S$ is demi-normal.
\end{proposition}
\begin{proof}
First let us assume that $S$ is Cartier. Let $\nu\colon (\bar{X},\bar{S}+\bar{\Delta}+\bar{D})\to (X,S+\Delta)$ be the normalization morphism, where $\bar{D}\subset \bar{X}$ is the conductor divisor. By the classification of lc surface singularities we know that $\bar{S}$ is nodal in codimension one, and by \cite[Theorem 1.5]{ABP} we know it is $S_2$ (see also \cite[Cor. 23]{Arv22}). Thus $\bar{S}$ is demi-normal, and by adjunction the boundary $(\bar{\Delta}+\bar{D})|_{\bar{S}}$ gives to $\bar{S}$ a structure of slc pair. We claim this implies that $S$ is also demi-normal. This is a local question, so we may localize on $X$ and assume that $X=\Spec A$ and $S=V(s)$ for some $s\in A$. The normalization $\bar{X}$ is also affine, say $\bar{X}=\Spec \bar{A}$, and $\bar{S}=V(s)\subset \bar{X}$. The section $s$ defines a morphism to $\bA^1$ on both $X$ and $\bar{X}$, fitting into the commutative diagram
        $$\begin{tikzcd}
        \bar{X}\arrow[dr, "\bar{\varphi}" below left]\arrow[rr,"\nu"] && X\arrow[dl, "\varphi"]\\
        & \bA^1. &
        \end{tikzcd}$$
Notice that $\varphi^{-1}(0)=S$ and $\bar{\varphi}^{-1}(0)=\bar{S}$. Localizing at $0\in\bA^1$, we may assume that the target of $\varphi$ and $\bar{\varphi}$ is the spectrum of a DVR. 
Then we are in situation to apply \cite[4.2.6]{Posva_Gluing_for_surfaces_in_mixed_char}, which concludes the Cartier case.

Next assume that $S$ is $\bQ$-Cartier, with Cartier index $m>0$ not divisible by the residue characteristics. We have that $S$ is nodal in codimension one by \autoref{lemma:slc_surface_codim_1} below. To show that $S$ is $S_2$ is a local property, so we may localize at any closed point. In particular we may assume that $\sO(-mS)\cong\sO_X$. Let $\pi\colon X'\to X$ be the cyclic cover induced by this isomorphism \cite[2.52-53 and 5.19]{Kollar_Mori_Birational_geometry_of_algebraic_varieties}. The scheme $X'$ is endowed with a $\mu_m$-action and the morphism $\mu_m$ is a geometric quotient. By assumption on $m$, the morphism $\pi$ is \'{e}tale in codimension one. So $X'$ is demi-normal. If $S'=\red(\pi^{-1}S)$ and $\Delta'=\red(\pi^{-1}\Delta)$ then $\pi^*(K_X+S+\Delta)=K_{X'}+S'+\Delta'$ and $S'$ is Cartier. So by \cite[2.43]{Kollar_Singularities_of_the_minimal_model_program} the pair $(X',S'+\Delta')$ is slc and by the first part of the proof $\sO_{S'}$ is $S_2$. Now observe that
        $$\pi_*\sI_{S'}=\bigoplus_{i=0}^{m-1}\sO(iS)[\otimes]\sO(-S).$$
The $\mu_m$-coaction on $\sO_{X'}=\bigoplus_{i=0}^{m-1}\sO(iS)$ makes $\pi_*\sI_{S'}$ into an $\mu_m$-equivariant sheaf, with invariant part $\sO(-S)=\sI_S$. Since $\mu_m$ is diagonalizable, taking invariants is an exact functor (see for example \cite[I.2.5, I.2.11]{Jantzen_Representation_algebraic_groups} for basics about diagonalizable group schemes). Thus the exact sequence
        $$0\to\sI_{S'}\to\sO_{X'}\to\sO_{S'}\to 0$$
yields the exact sequence
        $$0\to \sI_S\to \sO_X\to \left(\sO_{S'}\right)^{\mu_m}\to 0.$$
From this we see that $\left(\sO_{S'}\right)^{\mu_m}=\sO_S$. Since $(\sO_{S'})^{\mu_m}$ is a direct summand of $\sO_{S'}$ \cite[I.2.11]{Jantzen_Representation_algebraic_groups}, we deduce that $\sO_S$ is $S_2$. This concludes the proof.
\end{proof}

\begin{lemma}\label{lemma:slc_surface_codim_1}
Let $(X,S+\Delta)$ be a two-dimensional excellent slc pair, where $S$ is a reduced divisor. Then $S$ is a nodal curve.
\end{lemma}
\begin{proof}
Let $\nu\colon (\bar{X},\bar{D}+\bar{S}+\bar{\Delta},\tau)\to (X,S+\Delta)$ be the normalization. Then $\bar{D}$ is a nodal curve and $\bar{S}$ does not pass through the nodes $p_1,\dots,p_n$ of $D$ \cite[2.31]{Kollar_Singularities_of_the_minimal_model_program}. This implies that $S$ does not contain the images $\nu(p_i)$. So we may restrict on $X\setminus\{\pi(p_1),\dots,\pi(p_n)\}$ and assume that the curve $\bar{D}$ is normal. In particular $\tau$ is an involution on $\bar{D}$. Thus the morphism $\bar{X}\to X$ is the quotient by the finite equivalence relation $R(\tau)$ where the equivalence classes have cardinality at most $2$. Moreover $\bar{D}$ does not pass through the nodes of $\bar{S}$, by \cite[2.31]{Kollar_Singularities_of_the_minimal_model_program} again. Therefore $\bar{S}\to S$ is obtained by gluing pairs of smooth points of $\bar{S}$. It follows that $S$ has at worst nodal singularities.
\end{proof}

\subsection{Semi-normality of unions of lc centers in large characteristics}

\begin{theorem}\label{Mainsemi} Let $p_0$ be the positive integer occurring in \autoref{MainvanII}. Let $\Delta$ be a boundary with standard coefficients. Let $(X,\Delta)$ be a log canonical threefold whose closed points have perfect residue fields of characteristic $p>p_0$. Then the union of all log canonical centers of $(X,\Delta)$ is semi-normal.
\end{theorem}

\begin{proof}
The proof follows similar lines as \autoref{maintheorem}, however using instead the vanishing theorem of \autoref{MainvanII}. Let $C$ denote the union of all log canonical centers of $(X,\Delta_X)$.
Let $g\colon (Z, \Delta_Z+E)\to (X,\Delta_X)$ be the dlt model described in \autoref{MainvanII}. Since $(Z, \Delta_Z+ E)$ is dlt \cite[Theorem 19]{Bernasconi_Kollar_vanishing_theorems_for_threefolds_in_char_>5} implies that $E$ is $S_2$. Therefore $E$ is weakly normal by \autoref{lemma:divisor_is_weakly_normal}.   Consider the short exact sequence associated to $E$ inside $Z$:
$$\begin{tikzcd} 0\arrow{r}& \mathcal{O}_Z(-E)\arrow{r} &\mathcal{O}_Z\arrow{r} &\mathcal{O}_E\arrow{r} & 0.\end{tikzcd}$$
By \autoref{MainvanII} it holds that $R^1g_*\mathcal{O}_Z(-E)=0$ and hence, pushing forward along $g$, on $X$ there is a short exact sequence
$$\begin{tikzcd} 0\arrow{r}& g_*\mathcal{O}_Z(-E)\arrow{r} &g_*\mathcal{O}_Z \arrow{r} &g_*\mathcal{O}_E\arrow{r} &0. \end{tikzcd}$$
 We have $g_*\mathcal{O}_Z=\mathcal{O}_X$ and $g_*\mathcal{O}_Z(-E)=\mathcal{I}_C$ is the ideal sheaf of $C$. Let $\varphi\colon C^{sn}\to C$ denote the semi-normalization. 
 Since $E$ is semi-normal, the universal property of the semi-normalization \cite[I.7.2.3.3]{Kollar_Rational_curves} implies that there exists a factorization $\mathcal{O}_C\hookrightarrow \varphi_*\mathcal{O}_{C^{sn}}\hookrightarrow g_*\mathcal{O}_E$. Therefore we have the commutative diagram
 $$\begin{tikzcd}&&\mathcal{O}_C\arrow[hook]{rd}\\
 &&&\varphi_*\mathcal{O}_{C^{s\nu}}\arrow[hook]{d}\\
 \mathcal{O}_X\arrow[twoheadrightarrow]{rruu}\arrow[twoheadrightarrow]{rrr}&&&g_*\mathcal{O}_E.
 \end{tikzcd}$$
    Since $\sO_X\twoheadrightarrow g_*\sO_E$ is surjective, we see that $\mathcal{O}_C\hookrightarrow \varphi_*\mathcal{O}_{C^{s\nu}}$ is also a surjection. Therefore $C\cong C^{sn}$, as claimed.
    \end{proof}

    
\section{Non-seminormal lc center in characteristic $3$}\label{section:counterexample}
Our starting point is the surface constructed in \cite{Bernasconi_Kawamata-Viehweg_vanishing_fails}, whose notation we follow. Let $k$ be an algebraically closed field of characteristic $3$. Then there exists a normal connected projective surface $T$ over $k$ such that:
    \begin{enumerate}
        \item $T$ is a projective $\bQ$-factorial Picard rank one del Pezzo surface;
        \item on $T$ there exists a proper curve $E_1$ such that $-K_T\equiv E_1$;
        \item on $T$ there is an ample $\bQ$-Cartier $\bZ$-divisor $A$ such that $h^1(T,\sO_T(-A))\neq 0$;
        \item $A\equiv -K_T$.
    \end{enumerate}
The pair $(T,E_1)$ is constructed in \cite[\S 3.1]{Bernasconi_Kawamata-Viehweg_vanishing_fails}. The ample divisor $A$ is defined in \cite[\S 3.2]{Bernasconi_Kawamata-Viehweg_vanishing_fails} and the non-vanishing is shown in \cite[Theorem 3.6]{Bernasconi_Kawamata-Viehweg_vanishing_fails}. The fact that $A\equiv -K_T$ follows from \cite[3.2]{Bernasconi_Kawamata-Viehweg_vanishing_fails}.

\begin{claim}
The pair $(T,E_1)$ is a plt log Calabi--Yau surface pair.
\end{claim}
\begin{proof}
Let $\Psi\colon S\to T$ denote the minimal resolution of $T$. An easy computation shows that, in the notations of \cite[\S 3.2]{Bernasconi_Kawamata-Viehweg_vanishing_fails},  
    $$\psi^*(E_1)=\psi^{-1}_*E_1+\frac{1}{3}C+\frac{2}{3}H_1+\frac{1}{3}G_1+\frac{1}{3}F_1.$$ 
Since $\psi^*(K_T)=K_S+\sum_{i=1}^3\frac{1}{3}F_i+\frac{1}{3}C$ \cite[Proof of Proposition 3.2]{Bernasconi_Kawamata-Viehweg_vanishing_fails}, it follows that $(T,E_1)$ is plt. Since $K_T+E_1$ is numerically trivial, it follows that $K_T+E_1\sim_{\bQ}0$.
\end{proof}

Now let us define the divisors $L=K_T+A$ and $B=L+E_1$.

\begin{claim}\label{claim:numerically_trivial_div}
$L$ is numerically trivial and $h^1(T,\sO_T(L))\neq 0$.
\end{claim}
\begin{proof}
Numerical triviality holds by construction, and the non-vanishing holds by Serre duality.
\end{proof}

\begin{proposition}\label{prop:non_injectivity_H^1}
$B$ is ample. Moreover $h^1(T,\sO_T(B))=0$ and $h^1(T,\sO_T(B-E_1))\neq 0$.
\end{proposition}
\begin{proof}
Since $L$ is numerically trivial and $E_1$ is ample, $B$ is ample. We have $h^1(T,\sO_T(B-E_1))=h^1(T,\sO_T(L))\neq 0$ by \autoref{claim:numerically_trivial_div}. If $E_2$ and $E_3$ are the two other curves on $T$ involved in the definition of $A$, then 
$B=K_T+E_2+E_3$ \cite[\S 3.2]{Bernasconi_Kawamata-Viehweg_vanishing_fails}. By Serre duality $h^1(T,\sO_T(B))= h^1(T,\sO_T(-E_2-E_3))$, and the latter $h^1$ is $0$ by \cite[Proposition 3.3]{Cascini_Tanaka_Plt_threefolds_with_non_normal_centres_in_char_2}.
\end{proof}

\begin{proposition}
The pair $(\aC(T,B),E_{1,\aC(T,B)})$ is lc, and $E_{1,\aC(T,B)}$ is a non-minimal lc center.
\end{proposition}
\begin{proof}
The pair $(\aC(T,B),E_{1,\aC(T,B)})$ is lc by \autoref{Sing 3fold Cones}, and obviously $E_{1,\aC(T,B)}$ is an lc center. However it is not minimal. Consider the partial resolution $f\colon \aBC(T,B)\to \aC(T,B)$ with exceptional divisor $T^-$. As $K_X+E\sim_{\bQ} 0$, by \cite[(2.4.2)]{Bernasconi_Kawamata-Viehweg_vanishing_fails} we have
        $$K_{\aBC(T,B)}+f^{-1}_*E_{1,\aC(T,B)}+T^-=
        f^*(K_{\aC(T,B)}+E_{1,\aC(T,B)}).$$
Therefore $T^-$ is an lc place of $(\aC(T,B),E_{1,\aC(T,B)})$ and the minimal lc center is the vertex of the cone.
\end{proof}

\begin{proposition}
$E_{1,\aC(T,B)}$ is not $S_2$ nor seminormal.
\end{proposition}
\begin{proof}
The $S_2$ condition fails by \autoref{prop:S_2_at_vertex}, since by \autoref{prop:non_injectivity_H^1} the natural map $H^1(T,\sO_T(B-E_1))\to H^1(T,\sO_T(B))$ is not injective. 
We can say more: consider the exact sequence of cohomology
        $$H^0(T,\sO_T(dB))\overset{\psi_d}{\longrightarrow}H^0(E_1,\sE_d)
        \to H^1(T,\sO_T(dB-E_1))\overset{\alpha_d}{\longrightarrow} H^1(T,\sO_T(dB))$$
obtained from the restriction sequences \autoref{eqn:restrictrion_exact_seq} in our setting. Let $E'=\Spec_k \bigoplus H^0(E,\sE_d)$ with vertex $v'$. Then by \autoref{corollary:partial_normalization_cone} we have a finite homeomorphism $\pi\colon E'\to E_{1,\aC(T,B)}$ which is an isomorphism outside the vertex and induces an isomorphism $k(v)\cong k(v')$. Now $\alpha_1$ is not injective by \autoref{prop:non_injectivity_H^1}, so $\psi_1$ is not surjective: this implies that the inclusion $\bigoplus_{d\geq 0}\im\psi_d\subset \bigoplus H^0(E,\sE_d)$ is not an equality. This implies in turn that $\pi\colon E'\to E_{1,\aC(T,B)}$ is not an isomorphism, and therefore $E_{1,\aC(T,B)}$ is not seminormal.
\end{proof}

\begin{corollary}\label{corollary:counterex}There exists a log canonical threefold singularity $(y\in Y,\Delta_Y)$ over a perfect field of characteristic $3$ such that the union of all log canonical centers of $(Y,\Delta_Y)$ is a divisor which is not semi-normal. 
\end{corollary}
    \begin{proof}We let $Y=\aC(T,B)$ and $\Delta_Y=E_{1,\aC(T,B)}$, in the notation of this subsection. We consider $y\in Y$ to be the cone point. We have already showed that $(Y,\Delta_Y)$ is log canonical around $y$ and that $\Delta_Y$ is a log canonical center of this pair. We will show that any other log canonical center of $(Y,\Delta_Y)$ is contained in $\Delta_Y$. Suppose in order to arrive at a contradiction that $C\subset Y$ is a log canonical center of $(Y,\Delta_Y)$ and that $C$ is not contained in $\Delta_Y$. Since we work in a neighbourhood of $y$ we may assume that $y\in C$.
    
    As already mentioned in the proof of \autoref{lc klt assuming inv adjunction}, the morphism 
            $$f\colon (\aBC(T,B),E_{\aBC(T,B)}+(1+r)T^-)\longrightarrow
            (\aC(T,B),E_{\aC(T,B)})$$
    is crepant, where $r\in\bQ$ is such that $K_T+E\sim_{\bQ}rB$. By \autoref{claim:numerically_trivial_div} we have $r=0$. Thus $(\aBC(T,B),E_{\aBC(T,B)}+T^-)$ has a log canonical center $C'$ such that $f(C')=C$. Since $C\neq \{y\}$ we have $C'\nsubseteq T^-$, and since $C\nsubseteq E_{\aC(T,B)}$ we have $C'\nsubseteq E_{\aBC(T,B)}$. 
    
    Since $T$ is $\mathbb{Q}$-factorial, the cone $\aBC(T,B)$ is $\mathbb{Q}$-factorial \cite[Proof of Proposition 2.4]{Bernasconi_Kawamata-Viehweg_vanishing_fails}. So if $E$ is an lc place over $(\aBC(T,B),E_{\aBC(T,B)}+T^-)$ with center $C'$, we have the equality of discrepancies
            $$-1=a(E;\aBC(T,B), E_{\aBC(T,B)}+T^-)=
            a(E;\aBC(T,B),T^-)$$
    by \cite[Lemma 2.27]{Kollar_Mori_Birational_geometry_of_algebraic_varieties}. This shows that $C'$ is also a log canonical center of $(\aBC(T,B),T^-)$.
    
    Now we apply inversion of adjunction for $(\aBC(T,B),T^-)$ along $T^-$. Notice that $K_T\sim_\bQ -A$ is anti-ample, so $K_T\sim_\bQ rB$ for some $r\in\bQ_{-}$. So the desired inversion of adjunction holds by \autoref{remark:inv_adj_for_threefolds}. As $T$ is klt we obtain by \autoref{lc klt assuming inv adjunction}
    that $(\aBC(T,B),T^-)$ is plt in a neighbourhood of $T^-$. This fact contradicts what we obtained above, namely that $C'\neq T^-$ is also a log canonical center of $(\aBC(T,B),T^-)$. We have reached a contradiction, and therefore the proof is complete.
    \end{proof}


\section{Sufficient conditions for a non (semi-)normal boundary}\label{section:sufficient_conditions}
To conclude this note, we describe two sets of conditions that are sufficient to obtain examples of cone pairs with non-(semi-)normal reduced boundary. In both cases it boils down to a single cohomology statement on the base of the cone. Our criteria are independent of the dimension.

Let us fix the notations. The assumption on the Picard rank is not strictly speaking necessary, but it makes our conditions easier to state.

\begin{notation}\label{notation:counterexample}
We let $X$ be a projective normal variety of Picard rank one over an algebraically closed field and $E$ a prime effective $\bQ$-Cartier $\bZ$-divisor on $X$. We assume that $(X,E)$ is lc and that $K_X+E$ is numerically trivial or anti-ample.
\end{notation}

 If $A$ is an ample $\bQ$-Cartier $\bZ$-divisor divisor on $X$, then we have $K_X+E\sim_{\bQ}-rA$ for some positive rational number $r\in\bQ_{\geq 0}$. So $(\aC(X,A),E_{\aC(X,A)})$ is a pair \cite[2.4]{Bernasconi_Kawamata-Viehweg_vanishing_fails}. If we additionally assume inversion of adjunction in the form of \autoref{inversion} then it follows that this pair is actually lc by \autoref{lc klt assuming inv adjunction}. Below we give two sets of sufficient conditions for $E_{\aC(X,A)}$ to be non-$S_2$ and non-seminormal at the vertex.

\subsection{Failure of Kawamata--Viehweg vanishing for a nef divisor}

The first set of sufficient conditions is the one that we have applied to get our counter-example in characteristic $3$.

\begin{proposition}\label{prop:suffi_cond_1}
Situation as in \autoref{notation:counterexample}. Assume that there exists a 
nef $\bQ$-Cartier $\bZ$-divisor $L$ such that $H^1(X,\sO(L))\neq 0$. Then there exists $n\geq 1$ such that the cone boundary $E_{\aC(X,L+nE)}$ is not $S_2$ nor seminormal.
\end{proposition}

Indeed, take $n$ minimal such that 
        $$H^1(S,\sO(L+nE))=0 \quad \text{but}\quad H^1(S,\sO(L+(n-1)E))\neq 0.$$
By assumption on $L$ we have $n\geq 1$. Then $A=L+nE$ is ample, and it follows from \autoref{prop:S_2_at_vertex} and \autoref{corollary:partial_normalization_cone} that $E_{\aC(X,A)}$ is not $S_2$ nor seminormal.

\subsection{Failure of Kawamata--Viehweg for an anti-ample divisor}

The second set of sufficient conditions grew out of our efforts to find an alternative to \autoref{prop:suffi_cond_1}.

\begin{proposition}\label{prop:failure_KVV_anti_ample}
Situation as in \autoref{notation:counterexample}. Assume that there exists an ample $\bQ$-Cartier $\bZ$-divisor $A$ with the following properties:
    \begin{enumerate}
        \item $H^1(X,\sO(-A))\neq 0$, and
        \item $A-(\ind-1)(E)\cdot E$ is ample, where $\ind(E)$ is the Cartier index of $E$.
    \end{enumerate}
Then there exist $d>0$ and $0\leq i\leq \ind(E)-1$ such that $B=dA-iE$ is ample and the cone boundary $E_{\aC(X,B)}$ is not $S_2$ nor seminormal.
\end{proposition}

The second condition means that $E$ is numerically small compared to $A$. For the del Pezzo surface studied in \cite{ABL} this criterion cannot be applied, since the ample $A$ belongs to the smallest effective numerical class.

The proof goes as follows. First, consider the cone pair $(\aC(X,A),E_{\aC(X,A)})$. By the first condition the cone $\aC(X,A)$ is not $S_3$ at the vertex point (see \cite[4.2]{Bernasconi_Kawamata-Viehweg_vanishing_fails}). By \cite[2.4]{Bernasconi_Kawamata-Viehweg_vanishing_fails} we have $\ind(E_{\aC(X,A)})=\ind(E)=m$. If $E$ is Cartier then $E_{\aC(X,A)}$ is Cartier and we obtain that $E_{\aC(X,A)}$ cannot be $S_2$. By \autoref{prop:S_2_at_vertex} and \autoref{corollary:partial_normalization_cone} we deduce that $E_{\aC(X,A)}$ is not seminormal either. Thus we may assume that $m>1$ for the rest of the proof.

Denote by $D$ the divisor $E_{\aC(X,A)}$ localized at the vertex $v$. We have an isomorphism $\sO(mD)\cong \sO_{\aC(X,A),v}=\sO$. This isomorphism gives a structure of $\sO$-algebra to $\bigoplus_{i=0}^{m-1}\sO(-iD)$. Let
        $$\pi\colon X'=\Spec_{\sO}\bigoplus_{i=0}^{m-1}\sO(-iD)\longrightarrow \Spec \sO$$
be the induced finite morphism (see \cite[2.52]{Kollar_Mori_Birational_geometry_of_algebraic_varieties}). Then $X'$ is an integral $S_2+G_1$ semi-local scheme. Denote by $D'$ the divisorial pullback of $D$: it holds that $\pi^{[*]}\sO(D)=\sO_{X'}(D')$. By \cite[2.53]{Kollar_Mori_Birational_geometry_of_algebraic_varieties} we see that $\sO_{X'}(D')$ is Cartier (actually principal).

We claim that the reduced subscheme $D'\subset X'$ cannot be $S_2$. Suppose for the sake of contradiction that it is. Then $X'$ would be $S_3$ at its closed points by \cite[5.3]{Kollar_Mori_Birational_geometry_of_algebraic_varieties}, hence $S_3$ everywhere. By \cite[5.4]{Kollar_Mori_Birational_geometry_of_algebraic_varieties} it would follow that $\bigoplus_{i=0}^{m-1}\sO(-iD)$ is an $S_3$ $\sO$-module. But $\sO$ is a direct summand, so $\sO$ would be $S_3$. In other words $\aC(X,A)$ is $S_3$ at its vertex, which we know not to be the case.

Thus $\sO_{D'}$ is not $S_2$, and neither is its pushforward $\pi_*\sO(D')$ as $\sO$-module. Let us identify $\pi_*\sO_{D'}$. We have the exact sequences
        \begin{equation*}
            0\to \sO(-(i+1)D)\to \sO(-iD)\to \left(\sO(-iD)\otimes\sO_{E_{\aC(X,A)}}\right)/\text{tors}\to 0
        \end{equation*}
and for ease of notation we write $\sB_i$ the cokernel. Then $\sB_i$ is a torsion-free rank $1$ $\sO_{E_{\aC(X,A)},v}$-module. Since $\pi_*\sO_{X'}(-D')=\bigoplus_{i=0}^{m-1}\sO(-iD)[\otimes]\sO_{E_{\aC(X,A)}}$ we obtain that $\pi_*\sO_{D'}=\bigoplus_{i=0}^{m-1}\sB_i$. We know that $\pi_*\sO(D')$ is not $S_2$, thus some $\sB_i$ is not $S_2$ as $\sO$-module. We fix this index $i$ for the rest of the proof. 

Next we want to a description of the non-zero local cohomology group $H^1_v(\sB_i)$. We have the exact sequence
        \begin{equation}\label{eqn:H^1_v_cokernel}
            0=H^0_v(\sB_i)\to H^0(\sB_i)\to H^0\left( U\cap E_{\aC(X,A)}, \sB_i\right)\to H^1_v(\sB_i)\to H^1(\sB_i)=0.
        \end{equation}
Let us identify the first two non-zero groups in this sequence:
    \begin{itemize}
        \item On the affine $\aC(X,A)$ we have the exact sequence
                $$
                0\to H^0\left( E_{\aC(X,A)},\sO(-(i+1)E_{\aC(X,A)}\right)\to H^0\left(\sO(-iE_{\aC(X,A)}\right)\to H^0(\sB_i)
                \to 0.
                $$
        As for any $j\in\bZ$ we have
            \begin{eqnarray*}
            H^0\left( \sO(jE_{\aC(X,A)}\right) &=&
            H^0\left( \aBC(X,A),\sO(jE_{\aBC(X,A)})\right)\\
            &=&\bigoplus_{d\geq 0}H^0(X,\sO(dA+jE))
            \end{eqnarray*}
        we see that
            $$H^0(\sB_i)=\bigoplus_{d\geq 0}\coker\left[ 
            H^0(\sO_X(dA-(i+1)E))\longrightarrow H^0(\sO_X(dA-iE))
            \right].$$
        \item The isomorphism $U\cong \aBC(X,A)\setminus X^-=\Spec_X\bigoplus_{d\in\bZ}\sO_X(dA)$ induces an isomorphism
                $$\sB_i|_{E_{\aC(X,A)}\cap U}\cong \bigoplus_{d\in\bZ}\coker\left[ 
                \sO_X(dA-(i+1)E)\longrightarrow \sO_X(dA-iE)
                \right].$$
        Thus if for any $d,j\in\bZ$ we let the sheaf $\sE_{d,j}$ on $E$ be defined by
                $$\sE_{d,j}=\coker\left[ 
                \sO_X(dA-(j+1)E)\longrightarrow \sO_X(dA-jE)
                \right],$$
        we obtain
                $$H^0\left( U\cap E_{\aC(X,A)}, \sB_i\right) =\bigoplus_{d\in\bZ}H^0(E,\sE_{d,-i}).$$
    \end{itemize}
Now consider the exact sequences of cohomology groups on $X$
        $$H^0(\sO(dA-iE))\overset{\psi_{d,-i}}{\longrightarrow} H^0(E,\sE_{d,-i})
        \to H^1(\sO(dA-(i+1)E))
        \overset{\alpha_{d,-i}}{\longrightarrow}
        H^1(\sO(dA-iE)).$$
It follows from the exact sequence \autoref{eqn:H^1_v_cokernel} and from the explicit descriptions of the first two terms of the said exact sequence, that
        $$H^1_v(\sB_i)\neq 0\quad \Leftrightarrow 
        \quad \exists d\in\bZ: \psi_{d,-i} \text{ not surjective}
        \quad \Leftrightarrow \quad  \exists d\in\bZ: \alpha_{d,-i}\text{ not injective}.$$
Now $A$ and $E$ are ample, so $H^0(E,\sE_{d,-i})=0$ for $d\leq 0$. (The case $d=0=i$ follows simply from that $H^0(X,\sO_X)\cong H^0(E,\sO_E)$ are equal to the algebraically closed base-field). Thus
        $$H^1_v(\sB_i)\neq 0 \quad \Leftrightarrow \quad \exists d_i>0: \alpha_{d_i,-i}\text{ not injective}.$$
Since we chose $\sB_i$ with the property that $H^1_v(\sB_i)\neq 0$, we would like to take as new ample divisor $B=d_iA-iE$. In general $B$ need not be ample. But if the second hypothesis of \autoref{prop:failure_KVV_anti_ample} is satisfied, as $0\leq i<\ind(E_{\aC(X,A)})=\ind(E)$, we obtain that $B$ is ample. In this case we conclude by \autoref{prop:S_2_at_vertex} and \autoref{corollary:partial_normalization_cone}.

\bibliographystyle{alpha}
\bibliography{Bibliography}

\end{document}